\documentclass[12pt]{amsart}
\usepackage[latin1]{inputenc}
\usepackage[english]{babel}

\usepackage{indentfirst}
\usepackage{amssymb}
\usepackage{amsthm}
\usepackage{color}

\newcommand{\con}{\mathfrak c}
\newcommand{\FF}{\protect{\mathcal F}}

\newcommand{\sm}{\setminus}
\newcommand{\sub}{\subseteq}

\def\cA{{{\mathcal A}}}

\def\cC{{{\mathcal C}}}

\def\cD{{{\mathcal D}}}

\def\cF{{{\mathcal F}}}

\def\cZ{{{\mathcal Z}}}

\def\natnums{\mathbb N}

\def\N{\natnums}

\newcommand\eps{\ensuremath{\varepsilon}}

\newcommand{\Ba}{{\rm Ba}}

\newcommand{\impli}{\Rightarrow}
\newcommand{\Nat}{\mathbb{N}}

\newcommand{\erre}{\mathbb{R}}

\newcommand{\alg}{\mathfrak A}
\newcommand{\algb}{\mathfrak B}
\newcommand{\algc}{\mathfrak C}

\newcommand{\ult}{{\rm ULT}}

\newcommand{\en}{\mathbb N}

\newcommand{\clop}{\protect{\rm Clop} }
\newcommand{\0}{\protect{\bf 0}}
\newcommand{\1}{\protect{\bf 1}}

\newtheorem{theo}{Theorem}[section]
\newtheorem{lem}[theo]{Lemma}%[section]
\newtheorem{cor}[theo]{Corollary}%[theo]
\newtheorem{defi}[theo]{Definition}%[section]
\newtheorem{fact}[theo]{Fact}%[section]

\newtheorem{lemma}[theo]{Lemma}

\theoremstyle{definition}

\theoremstyle{remark}
\newtheorem{remark}[theo]{Remark}
\numberwithin{equation}{section}

\def\epsilon{\varepsilon}

\providecommand{\MR}{\relax\ifhmode\unskip\space\fi MR }
\providecommand{\href}[2]{#2}

\title{A weak$^*$ separable $C(K)^*$ space whose
unit ball is not weak$^*$ separable}

\author{ A. Avil\'{e}s}
\address{Departamento de Matem\'{a}ticas\\
Facultad de Matem\'{a}ticas\\ Universidad de Murcia\\ 30100 Espinardo (Murcia)\\
Spain} \email{avileslo@um.es}

\author{G. Plebanek}
\address{Instytut Matematyczny\\ Uniwersytet Wroc\l awski\\ Wroc\l aw\\ Poland} \email{grzes@math.uni.wroc.pl}

\author{J. Rodr\'{i}guez}
\address{Departamento de Matem\'{a}tica Aplicada\\
Facultad de Inform\'{a}tica\\ Universidad de Murcia\\ 30100 Espinardo (Murcia)\\
Spain} \email{joserr@um.es}



\thanks{A.~Avil\'{e}s and J.~Rodr\'{i}guez were supported by MEC and FEDER (Project MTM2008-05396)
and Fundaci\'{o}n S\'{e}neca (Project 08848/PI/08). A. Avil\'{e}s was 
supported by {\em Ramon y Cajal} contract (RYC-2008-02051) and an FP7-PEOPLE-ERG-2008 action.
G. Plebanek was supported by MNiSW Grant N N201 418939 (2010--2013)}

\date{\today}

\begin{document}
\begin{abstract}
We provide a ZFC example of a compact space $K$ such that 
$C(K)^*$ is $w^*$-separable but its closed unit ball $B_{C(K)^*}$ is not $w^*$-separable.
All previous examples of such kind had been constructed under CH. We also discuss
the measurability of the supremum norm on that~$C(K)$ equipped with its weak Baire
$\sigma$-algebra.
\end{abstract}

\maketitle

\section{Introduction}

Let $K$ be a compact space (all our topological spaces are assumed to be Hausdorff) and let $C(K)$ be
the Banach space of all continuous real-valued functions on~$K$ (equipped with the supremum norm). 
One can consider the following list of properties related to separability in~$K$ and $C(K)^\ast$:
\begin{itemize}
\item[(a)] $K$ is separable;
\item[(b)] $K$ carries a strictly positive measure of countable type;
\item[(c)] $P(K)$ (the set of all Radon probability measures on~$K$) is $w^\ast$-separable;
\item[(d)] $B_{C(K)^\ast}$ (the closed unit ball of~$C(K)^*$) is $w^\ast$-separable;
\item[(e)] $C(K)^\ast$ is $w^\ast$-separable.
\end{itemize}
It is known that the following implications hold
\begin{center}
(a) $\Longrightarrow$ (b) $\Longrightarrow$ (c) $\Longleftrightarrow$ (d) $\Longrightarrow$ (e)
\end{center}
and cannot be reversed in general. Indeed, the Stone space of the measure algebra of the Lebesgue measure on~$[0,1]$ 
satisfies (b) but not~(a), while Talagrand~\cite{Talagrand80}
constructed under~CH two examples showing that (c)$\Rightarrow$(b) and (e)$\Rightarrow$(c) do not hold.
Recently, D\v{z}amonja and Plebanek~\cite{DzamonjaPlebanek08} gave a ZFC counterexample to (c)$\Rightarrow$(b). 

In this paper we provide a ZFC example of a compact space $K$ witnessing that the implication
(e)$\Rightarrow$(c) does not hold,
i.e. $C(K)^*$ is $w^*$-separable but $B_{C(K)^*}$ is not. The construction is given in Section~\ref{example} (see Theorem~\ref{ex2:main})
and uses techniques which are similar to those of~\cite{DzamonjaPlebanek08}.
 
Section~\ref{norm} is devoted to discuss further properties of that example which are relevant within the topic of measurability in Banach spaces. 
In every Banach space~$X$ one can consider the Baire $\sigma$-algebra of the weak topology, denoted by $\Ba(X,w)$, which coincides with
the $\sigma$-algebra on~$X$ generated by~$X^*$ (see \cite[Theorem~2.3]{edg-J}). While $\Ba(X,w)$ coincides with
the Borel $\sigma$-algebra of the norm topology if $X$ is separable, both $\sigma$-algebras 
may be different if $X$ is nonseparable. Given any equivalent norm $\|\cdot\|$ on~$X$, its closed unit ball $B_X$ belongs to
$\Ba(X,w)$ if and only if $\|\cdot\|$ is $\Ba(X,w)$-measurable (as a real-valued function on~$X$). 
It is easy to check that:
\begin{center}
$B_{X^*}$ is $w^*$-separable $\Longrightarrow$ $\|\cdot\|$ is $\Ba(X,w)$-measurable $\Longrightarrow$ $X^*$ is $w^*$-separable.
\end{center}
None of the reverse implications holds in general~\cite{rod9}. It seems to be an open problem whether
the weak Baire measurability of the {\em supremum norm} of a $C(K)$ space is equivalent to either $C(K)^\ast$ being $w^*$-separable, 
or to~$B_{C(K)^\ast}$ being $w^*$-separable. Our compact space~$K$ of Section~\ref{example} would settle one of the two questions in the negative, but we have been unable to determine the degree of measurability of the supremum norm on that~$C(K)$. 
We shall show, however, that at least there exists a norm dense set $E\sub C(K)$ where the restriction of the 
supremum norm is relatively $\Ba(C(K),w)$-measurable (Theorem~\ref{theo:XXX}). This set $E$ can be taken to be the linear span of the characteristic functions of clopen subsets of~$K$ under the set-theoretic assumption that $\mathfrak c$ is a Kunen cardinal (Theorem~\ref{analysing:6}).

\subsection*{Terminology}

We write $\mathcal{P}(S)$ to denote the power set of any set~$S$.
The cardinality of~$S$ is denoted by~$|S|$. The letter $\mathfrak{c}$ stands for the cardinality of the continuum.
A probability measure $\nu$ is said to be of countable type if the space $L^1(\nu)$ is separable.
For a compact space~$K$, we usually identify the dual space $C(K)^*$ with the space of all Radon measures on~$K$. 
Given a Boolean algebra~$\alg$, we write $\alg^+$ to denote
the set of all nonzero elements of~$\alg$. For the Boolean operations
we use the usual symbols $\cup$, $\cap$, etc. and we write $\0$ and $\1$ for the least and the greatest element.
The Stone space of all ultrafilters on~$\alg$ is denoted by~$\ult(\alg)$. Recall that the Stone isomorphism
between $\alg$ and the algebra $\clop(\ult(\alg))$ of clopen subsets of $\ult(\alg)$ is given by
$$
	\alg \to \clop(\ult(\alg)), \quad a \mapsto \widehat{a}=\{\FF\in\ult(\alg): \, a\in\FF\}.
$$
Every measure $\mu$ on~$\alg$ induces a measure $\widehat{a} \mapsto \mu(a)$ on $\clop(\ult(\alg))$ which can be uniquely 
extended to a Radon measure on~$\ult(\alg)$ (see e.g. \cite[Chapter~5]{semadeni}); 
such Radon measure is still denoted by the same letter~$\mu$.

Given a set $S$ and a family $\cD$ of subsets of~$S$, the $\sigma$-algebra on~$S$ generated by~$\cD$
is denoted by $\sigma(\cD)$. If $(\Omega,\Sigma)$ is any measurable space and $n\in \N$, we write
$\otimes_n \Sigma$ to denote the usual product $\sigma$-algebra on~$\Omega^n$, that is,
$$
	\otimes_n\Sigma=\sigma\bigl(\{A_1\times\dots\times A_n: \, A_i\in \Sigma \mbox{ for all }i=1,\dots,n\}\bigr).
$$
A cardinal $\kappa$ is called a {\em Kunen cardinal} if the equality 
$$
	\mathcal{P}(\kappa \times \kappa)=\mathcal{P}(\kappa)\otimes \mathcal{P}(\kappa)
$$
holds true. Of course, in this case we have $\mathcal{P}(\kappa^n)=\otimes_n\mathcal{P}(\kappa)$ for every $n\in \Nat$.
This notion was investigated by Kunen in his doctoral dissertation \cite{Kunen68}. It is known that:
\begin{enumerate}
\item[(i)] any Kunen cardinal is less than or equal to~$\mathfrak{c}$; 
\item[(ii)] $\omega_1$ is a Kunen cardinal; 
\item[(iii)] $\con$ is a Kunen cardinal under Martin's axiom, while it is relatively consistent that $\con$ 
is not a Kunen cardinal. 
\end{enumerate}
Kunen cardinals have been considered by Talagrand~\cite{tal10}, Fremlin~\cite{Fremlin80} and the authors~\cite{APR11}
in connection with measurability properties of Banach spaces, and also by Todorcevic~\cite{Todorcevic} in connection with 
universality properties of $\ell_\infty/c_0$.

\section{The example} \label{example}

Fix any cardinal $\kappa$ such that $\omega_1 \leq \kappa \leq \mathfrak{c}$.
Let $\lambda$ be the usual product probability measure on the Baire $\sigma$-algebra of~$2^{\kappa}$ and let $\algb$ be its measure algebra.
Note that $\algb$ has cardinality~$\con$ since every Baire subset of~$2^{\kappa}$ is determined by countably many coordinates
(see e.g. \cite[254M]{freMT-2}). 
The letter~$\lambda$ will also stand for the corresponding probability measure on~$\algb$.
We shall work in the countable simple product $\algc:=\algb^\N$ of $\algb$, so that 
every $c\in\algc$ is a sequence $c=(c(n))_n$ where $c(n)\in\algb$ for all $n\in \N$.
On the Boolean algebra $\algc$ we consider the sequence of probability measures $\{\mu_n:n\in \N\}$
defined by 
\[\mu_n(c):=\lambda(c(n)) \quad \mbox{for all } c\in\algc.\] 
Let $\{N_b: b\in \algb^+\}$ 
be a fixed {\em independent} family of subsets of~$\N$, i.e.
$$
	\bigcap_{b\in s} N_b \sm \bigcup_{b'\in t} N_{b'}\neq\emptyset
$$
whenever $s,t \sub \algb^+$ are finite and disjoint (see e.g. \cite[p.~180, Exercise 3.6.F]{Engelking}).
For each $b\in \algb^+$, define an element $G_b\in \algc$ by declaring
$$
	G_b(n):=
	\begin{cases}
	b & \text{if $n\in N_b$} \\
	\0 & \text{otherwise.}
	\end{cases}
$$
Let $\alg$ be the subalgebra of~$\algc$ generated by $\{G_b:b\in \algb^+\}$ and write $K:=\ult(\alg)$. This section is devoted to prove the following:

\begin{theo}\label{ex2:main}
$B_{C(K)^\ast}$ is not $w^\ast$-separable, while
$\{\mu_n:n\in \N\}$ separates the points of~$C(K)$, so
$C(K)^\ast$ is $w^\ast$-separable.
\end{theo}

In order to prove Theorem~\ref{ex2:main} we need some previous work. For any finite sets 
$s, t \sub \algb^+$ we consider the following elements of~$\alg$:
$$
	J(s):=\bigcap_{b\in s} G_b, \quad U(t):=\bigcup_{b\in t} G_b
	\quad\mbox{and} \quad 
	W(s,t):=J(s)\setminus U(t).
$$
The Boolean algebra $\alg$ is completely determined up to isomorphism when given its set of generators 
$\{G_b : b\in\algb^+\}$ and which elements $W(s,t)$ are zero. In this sense, Lemma~\ref{lem:zero} below is interpreted as the fact that $\alg$ is isomorphic to the free Boolean algebra generated by $\{G_b : b\in\algb^+\}$ modulo the relations 
that $W(s,t) = \0$ if and only if $s\cap t \neq \emptyset$ or $\bigcap_{b\in s}b = \0$.

\begin{lem}\label{lem:zero}
For two finite sets $s,t \sub \algb^+$, we have $W(s,t) = \0$ if and only if $s\cap t \neq\emptyset$ or $\bigcap_{b\in s}b=\0$. 
In particular, for a finite $s\sub \algb^+$ the following are equivalent:
\begin{enumerate}
\item $\bigcap_{b\in s}b=\0$;
\item $J(s)=\0$;
\item $W(s,t) = \0$ for all finite $t\sub \algb^+\setminus s$;
\item $W(s,t) = \0$ for some finite $t\sub \algb^+\setminus s$.
\end{enumerate}
\end{lem}
\begin{proof}
It is clear that if $s\cap t\neq \emptyset$ then $W(s,t)=\0$. On the other hand,
let us observe that for every $n\in \Nat$ we have
\begin{equation}\label{equation:JSn}
	J(s)(n)=\bigcap_{b\in s} G_b(n)=
	\begin{cases}
	\bigcap_{b\in s}b & \text{if $n\in \bigcap_{b\in s}N_b$} \\
	\0 & \text{otherwise.}
	\end{cases}
\end{equation}
So if $\bigcap_{b\in s}b=\0$, then $W(s,t)\sub J(s) = \0$ as well. For the converse,
suppose that $s\cap t=\emptyset$ and $\bigcap_{b\in s}b\neq \0$, and pick
$$
	n_0\in \bigcap_{b\in s}N_b \sm \bigcup_{b'\in t}N_{b'}.
$$
Then 
$$
	W(s,t)(n_0)=\bigcap_{b\in s}G_b(n_0) \sm \bigcup_{b'\in t}G_{b'}(n_0)=\bigcap_{b\in s}b 
	\neq \0
$$
and so $W(s,t)\neq \0$. The second part of the lemma, with the list of equivalences, follows from the first statement
and the arguments above.
\end{proof}

We next describe $K = \ult(\alg)$. Let us consider the family of all subsets of $\algb$ with the finite intersection property, 
that is
$$
	\mathrm{FIP}(\algb) = \left\{X\sub \algb^+: \, b_1\cap\ldots\cap b_n \neq\0 \, \mbox{ for every }b_1,\dots,b_n\in X\right\}.
$$
Given $X\in \mathrm{FIP}(\algb)$, let $\FF_X$ be the filter on~$\alg$ 
generated by 
$$
	\{W(s,t) : \, s\sub X \mbox{ finite}, \, t\sub \algb^+\sm X \mbox{ finite}\}
$$ 
(notice that this set has the finite intersection property by Lemma~\ref{lem:zero}).

\begin{lem}\label{lem:Fs}
$K = \left\{\mathcal{F}_X : X\in\mathrm{FIP}(\algb)\right\}$
\end{lem}
\begin{proof}
Every filter of the form $\FF_X$ is an ultrafilter on~$\alg$, 
because $\{G_b : b\in\algb^+\}$ is a set of generators of $\alg$ and,
for each $b\in\algb^+$, we have either $G_b = W(\{b\},\emptyset)\in \FF_X$ (if $b\in X$) or $\1\setminus G_b = W(\emptyset,\{b\})\in \FF_X$ (if $b\not\in X$). Conversely, let $\FF$ be any ultrafilter on~$\alg$ and consider $X := \{b\in\algb^+ : G_b\in \FF\}$. Notice that 
$X\in\mathrm{FIP}(\algb)$ and that, for each $b\in \algb^+$, we have
$$
	G_b\in \FF_X \quad \Longleftrightarrow \quad G_b\in \FF.
$$  
Since $\{G_b : b\in\algb^+\}$ is a set of generators of $\alg$, it follows that $\FF=\FF_X$. 
\end{proof}

Let $\mathrm{FIP}_{0}(\algb)$ be the family of all $s\in \mathrm{FIP}(\algb)$ which are finite. The next lemma says that the clopens of the form $\widehat{W(s,t)}$ are a basis of the topology of $K$, and also that $\{\FF_s : s\in \mathrm{FIP}_{0}(\algb)\}$ is dense in $K$. 

\begin{lem}\label{lem:density}
Let $a\in \alg^+$ and $\FF \in \widehat{a} \sub K$. 
\begin{enumerate}
\item[(i)] There exist $s\in  \mathrm{FIP}_{0}(\algb)$ and a finite set $t \sub \algb^+\sm s$ such that $\FF \in \widehat{W(s,t)} \sub \widehat{a}$.
\item[(ii)] If $\FF=\FF_{s_0}$ for some $s_0\in  \mathrm{FIP}_{0}(\algb)$, then we can choose $s=s_0$ in (i). 
\end{enumerate}
\end{lem}
\begin{proof}
(i) Take $I \sub \algb^+$ finite such that $a$ belongs to the subalgebra of~$\alg$ generated by $\{G_b:b\in I\}$.  Then 
$a$ can be written as the union of finitely many elements of the form $W(s,t)$, where $s\cup t=I$ and $s\cap t=\emptyset$. Since $a\in \FF$, there exist
$s$ and $t$ as before such that $W(s,t)\in \FF$, 
hence $s\in  \mathrm{FIP}_{0}(\algb)$ (by Lemma~\ref{lem:zero}) and $\FF \in \widehat{W(s,t)} \sub \widehat{a}$.

(ii) Since $W(s,t)\in \FF_{s_0}$, we have $G_b\in \FF_{s_0}$ for all $b\in s$ and $G_b\not \in \FF_{s_0}$ for all $b\in t$.
Hence $s \sub s_0$ and $s_0 \cap t=\emptyset$. Therefore,
$\FF_{s_0} \in \widehat{W(s_0,t)} \sub \widehat{W(s,t)} \sub \widehat{a}$.
\end{proof}

Another ingredient to prove Theorem~\ref{ex2:main} is the result isolated in Lemma~\ref{ex2:5} below, which
is a consequence of the following characterization of $w^*$-separability
in spaces of measures, due to M\"{a}gerl and Namioka~\cite{MagerlNamioka80}.

\begin{fact}\label{MagerlNamioka}
Let $L$ be a compact space. Then the space $P(L)$ is $w^\ast$-separable if and only if there is a 
sequence $\{\nu_n: n\in\en\}$ in~$P(L)$ such that for every nonempty open set $V\sub L$ there is $n\in \N$ 
with $\nu_n(V)> 1/2$. 
\end{fact}

\begin{lem}\label{ex2:5}
Let $\alg_1$ and $\alg_2$ be Boolean algebras such that
there is another Boolean algebra $\alg_3$ containing $\alg_1$ and $\alg_2$ as subalgebras and the following holds:
\begin{center}
($\star$) \ for every $b\in\alg_2^+$ there is $a\in\alg_1^+$ such that $a\sub b$. 
\end{center}
If $P(\ult(\alg_2))$ is not $w^*$-separable, then $P(\ult(\alg_1))$ is not $w^*$-separable either.
\end{lem}
\begin{proof}
Let $\{\nu_n: n\in\N\}$ be a sequence of probability measures on~$\alg_1$.
For each $n\in \N$, we can extend $\nu_n$ to a probability measure
$\nu_n'$ on $\alg_3$ (see~\cite{Lipecki74} or~\cite{Plachky76}) 
and we denote by $\nu_n$ the restriction of $\nu_n'$ to $\alg_2$.
Since $P(\ult(\alg_2))$ is not $w^*$-separable, by Fact~\ref{MagerlNamioka} there is $b\in\alg_2^+$ 
such that $\nu_n(b)\le 1/2$ for every $n\in \N$.
Property ($\star$) allows us to take $a\in\alg_1^+$ such that $a\sub b$. Then
\[ 
	\nu_n(a)=\nu_n'(a)\le \nu_n'(b)=\nu_n(b)\le 1/2
	\quad
	\mbox{for every }n\in \N.
\]
Another appeal to Fact~\ref{MagerlNamioka} ensures that
$P(\ult(\alg_1))$ is not $w^*$-separable.
\end{proof}

We are now ready to deal with Theorem~\ref{ex2:main}.

\begin{proof}[Proof of Theorem~\ref{ex2:main}]
According to a result of Rosenthal (see \cite[Theorem~3.1]{rosenthal}), the space $L^\infty(\nu)^*$ is not $w^*$-separable
whenever $\nu$ is a probability measure of uncountable type. This implies that
$P(\ult(\algb))$ is not $w^*$-separable
(bear in mind that $C(\ult(\algb))$ is isomorphic to~$L^\infty(\lambda)$). 
Let $\algb^*$ be the subalgebra of~$\algc$ consisting of all constant sequences. 
Since $\algb^*$ is isomorphic to~$\algb$, $P(\ult(\algb^*))$ is not $w^*$-separable. On the other hand,  
property ($\star$) of Lemma~\ref{ex2:5} holds for $\alg_1=\alg$ and~$\alg_2=\algb^*$, hence
$P(K)$ is not $w^*$-separable and so $B_{C(K)^*}$ is not $w^*$-separable either.

We now prove that $\{\mu_n:n\in \N\}$ separates the points of $C(K)$. 
Fix $h\in C(K)\setminus \{0\}$.

\smallskip
{\sc Step 1.}
Since $\{\FF_s:s\in  \mathrm{FIP}_{0}(\algb)\}$ is dense in~$K$ (by Lemma~\ref{lem:density}(i)),
there is $s\in  \mathrm{FIP}_{0}(\algb)$ such that $h(\FF_s)\neq 0$. Moreover, we can assume further that 
\begin{equation}\label{equation:minimality}
	h(\FF_{s'})=0 
	\quad\mbox{whenever }s'\in  \mathrm{FIP}_{0}(\algb) \mbox{ satisfies }|s'|<|s|
\end{equation}
and that $C:=h(\FF_s)>0$. 
By the continuity of~$h$ and Lemma~\ref{lem:density}(ii), there is
a finite set $t \sub \algb^+\sm s$ such that, writing $a:=W(s,t)$, we have
\begin{equation}\label{equation:positivity}
	h(\FF)\geq \frac{C}{2} \quad
	\mbox{for all }\FF\in \widehat{a}.
\end{equation}
Set $\delta:=\frac{C}{4}\lambda(\bigcap_{b\in s}b)>0$
and define $\mathcal{R}:=\{r\sub s: r\neq s\} \sub  \mathrm{FIP}_{0}(\algb)$.

\smallskip
{\sc Step 2.}
Fix $r\in \mathcal{R}$. Since $h$ is continuous
and $h(\FF_{r})=0$ (by~\eqref{equation:minimality}), we can apply Lemma~\ref{lem:density}(ii) 
to find a finite set $t'_r \sub \algb^+\sm r$ such that $|h(\FF)|\leq \delta$
for all $\FF \in \widehat{W(r,t'_r)}$. Writing $t_r:=t'_r\setminus s$, we have 
$$
	a_r:=W(r,t_r\cup (s\sm r)) \sub W(r,t'_r)  \quad \mbox{in }\alg
$$
and therefore
\begin{equation}\label{equation:controlR}
	|h(\FF)|\leq \delta \quad \mbox{for all }\FF \in \widehat{a_r}.
\end{equation}

\smallskip
{\sc Step 3.} Define
$$
	t^\ast: = t\cup \bigcup_{r\in \mathcal{R}} t_r \sub \algb^+\setminus s
$$ 
and choose $n\in \bigcap_{b\in s}N_b\sm\bigcup_{b'\in t^\ast} {N_{b'}}$
(which is nonempty by independence). Hence
$$
	a(n)=\bigcap_{b\in s}b, \qquad
	a_r(n)=\bigcap_{b\in r} b\sm \bigcup_{b'\in s\sm r}b'
	\quad \mbox{for every } r\in \mathcal{R},
$$
and therefore
\begin{equation}\label{equation:keyinclusion}
	\1\sm a(n) \sub \bigcup_{r\in \mathcal{R}} a_r(n) \quad \mbox{in }\algb. 
\end{equation}

\smallskip
{\sc Step 4.} Observe that
$$
	\mu_n\left( \bigl(K \setminus \widehat{a}\bigr) \sm \bigcup_{r\in \mathcal{R}}\widehat{a_r}\right)=
	\mu_n\left( \bigl(\1 \sm a\bigr) \sm \bigcup_{r\in \mathcal{R}}a_r \right) 
	=\lambda\left( \bigl(\1 \sm a(n)\bigr) \sm \bigcup_{r\in \mathcal{R}}a_r(n) \right)\stackrel{\eqref{equation:keyinclusion}}{=} 0.	
$$
Therefore
$$
	\left|\int_{K\setminus \widehat{a}}h \, d\mu_n \right|\leq 
	\int_{\bigcup_{r\in \mathcal{R}}\widehat{a_r}}|h| \, d\mu_n \stackrel{\eqref{equation:controlR}}{\leq} \delta.
$$
It follows that
\begin{multline*}
	\mu_n(h)=\int_{\widehat{a}}h \, d\mu_n + \int_{K \sm \widehat{a}}h \, d\mu_n \geq \\ \geq
	\int_{\widehat{a}}h \, d\mu_n - \delta \stackrel{\eqref{equation:positivity}}{\geq}
	\frac{C}{2}\mu_n(\widehat{a})-\delta = 
	\frac{C}{2}\lambda(a(n))-\delta =
	\frac{C}{2}\lambda\left(\bigcap_{b\in s}b\right)-\delta=\frac{C}{4}\lambda\left(\bigcap_{b\in s}b\right)>0.
\end{multline*}
Thus $\mu_n(h)\neq 0$. This finishes the proof.
\end{proof}

\begin{remark}
\rm We stress that Rosenthal's theorem used in the proof of Theorem~\ref{ex2:main} 
is a weakening of a result stating that $L^\infty(\nu)$ is not realcompact whenever $\nu$ is a probability measure
of uncountable type, see \cite{FrPl94}.
\end{remark}

\section{Weak Baire measurability of the norm}\label{norm}

Throughout this section we follow the notation introduced in Section~\ref{example}.
The supremum norm on~$C(K)$ is denoted by~$\|\cdot\|$.

We first show that the family $\{\mu_n:n\in\N\} \sub P(K)$, though separating the elements of~$C(K)$, is
not rich enough to ``measure'' $B_{C(K)}$. 

\begin{remark}\label{example-muN}
The supremum norm on $C(K)$ does not have to be measurable with respect to
the $\sigma$-algebra generated by $\{\mu_n:n\in\N\}$.
\end{remark}
\begin{proof}
We identify $\mathcal{P}(\N)$ and $\{0,1\}^{\N}$ in the usual way.
Let $\Omega\sub 2^\N$ be an independent family of subsets of~$\N$ with
$|\Omega|=\con$ and let $\Sigma$ denote the trace of ${\rm Borel}(2^\N)$ on~$\Omega$.
Then $|\Sigma|=\con$ and so we can find 
$A\sub\Omega$ such that $A\notin\Sigma$ and $|A|=|\Omega\sm A|=\con$.
Now we can choose an enumeration $\Omega=\{N_b:b\in \algb^+\}$ such that
\begin{equation}\label{equation:choiceA}
	\lambda(b)=1/2 \ \iff \ N_{b}\in A.
\end{equation}
Define a function $f:\Omega\to C(K)$ by
\[ 
	f(N_b):=\frac{1}{2\lambda(b)}1_{\widehat{G_{b}}}.
\]
We claim that $f$ is measurable with respect to~$\Sigma$ and 
the $\sigma$-algebra on $C(K)$ generated by $\{\mu_n:n\in \N\}$. Indeed, for fixed
$n\in\N$, we have
\[
	(\mu_n\circ f)(N_b)=\frac{1}{2\lambda(b)}\mu_n(G_{b})=
	\frac{1}{2\lambda(b)}\lambda(G_{b}(n))=
	\begin{cases}
		1/2 & \text{if }n\in N_{b}\\
		0 & \text{otherwise.}
	\end{cases}
\]
It follows that $\mu_n\circ f = (1/2)\pi_n$, where $\pi_n:\Omega \sub 2^\N\to \{0,1\}$ denotes the $n$-th coordinate
projection, hence $\mu_n\circ f$ is $\Sigma$-measurable.

On the other hand, the composition $\|f(\cdot)\|:\Omega \to \erre$ is not $\Sigma$-measurable because
\[
	\|f(N_b)\|=1 \ \stackrel{\eqref{equation:choiceA}}{\iff} \ N_b\in A
\]
and $A\notin\Sigma$. This implies that $\|\cdot\|$ is not measurable
with respect to the $\sigma$-algebra on~$C(K)$ generated by $\{\mu_n:n\in \N\}$.
\end{proof}

\subsection{Measurability on a norm dense set}\label{subsection:Sprime}

Let $S(K)$ denote the dense subspace of $C(K)$ consisting of all 
linear combinations of characteristic functions of clopen subsets of~$K$. The next simple lemma provides
a useful representation of the elements of~$S(K)$.

\begin{lem}\label{analysing:2}
Let $g\in S(K)$. Then there exist $s\sub \algb^+$ finite and a collection
of real numbers $\{y_r:r \sub s\}$ such that 
$$
	g=\sum_{r \sub s}y_r 1_{\widehat{J(r)}}.
$$
\end{lem}
\begin{proof}
We can write $g$ as
$$
	g=\sum_{r' \sub s}z_{r'}1_{\widehat{W(r',s\setminus r')}}
$$
for some $s\sub \algb^+$ finite and certain collection of real numbers
$\{z_{r'}: r' \sub s\}$. Note that there is a (unique) collection 
of real numbers $\{y_{r}: r \sub s\}$ such that
\begin{equation}\label{equation:yz}
	z_{r'}=\sum_{r \sub r'}y_r \quad
	\mbox{for every }r'\sub s.
\end{equation}
Since
$$
	J(r)=\bigcup_{r \sub r' \sub s}W(r',s \sm r')
	\quad \mbox{for every }r\sub s
$$
we conclude that
$$
	g=\sum_{r' \sub s}z_{r'}1_{\widehat{W(r',s\setminus r')}}
	\stackrel{\eqref{equation:yz}}{=}
	\sum_{r' \sub s}\sum_{r \sub r'}y_r 1_{\widehat{W(r',s\setminus r')}}= 
	\sum_{r \sub s}y_r \sum_{r \sub r' \sub s}1_{\widehat{W(r',s\setminus r')}}=
	\sum_{r \sub s}y_r 1_{\widehat{J(r)}}
$$
and the proof is over.
\end{proof}

We denote by $S'(K)$ the set of all $g\in S(K)$ which can be written as
$$
	g=\sum_{r \sub s}y_r 1_{\widehat{J(r)}}
$$
for some finite set $s\sub \algb^+$ and some collection of {\em nonzero} real numbers $\{y_r:r \sub s\}$.
It is easy to check (via Lemma~\ref{analysing:2}) that $S'(K)$ is norm dense in $S(K)$ and so, $S'(K)$ is norm
dense in $C(K)$. In Theorem~\ref{theo:XXX} we shall prove that the restriction of the supremum norm to~$S'(K)$ is measurable
with respect to the trace of $\Ba(C(K),w)$ on~$S'(K)$. The proof requires some work and the first 
step is to find a substantially larger family of measures to deal with.

\begin{lemma}\label{analysing:1}
For each $T \sub \algb^+$  and each $k\in \N$ there is a probability measure $\mu_T^k$ on~$\alg$ such that,
for every finite set $r\sub \algb^+$, we have 
\[ 
	\mu_T^k(J(r))=
	\begin{cases}
		\lambda\left(
		\bigcap_{b\in r}b\right)^k 
		& \text{if $r\sub T$},\\
		0 & \text{otherwise}.
	\end{cases}
\]
Moreover, for every finite disjoint sets $r,s\sub \algb^+$, we have 
\[ 
	\mu_T^1(W(r,s))=
	\begin{cases}
		\lambda\left(
		\bigcap_{b\in r}b \sm \bigcup_{b'\in s\cap T}b'\right)
		& \text{if $r\sub T$},\\
		0 & \text{otherwise}.
	\end{cases}
\]
Sometimes we shall write $\mu_T:=\mu_T^1$.
\end{lemma}
\begin{proof}
Let $\algb_k:=\algb\otimes\cdots\otimes\algb$ denote the free product of~$k$ many copies of~$\algb$ and 
let $\lambda_k$ denote the product measure on~$\algb_k$ (see e.g. \cite[2.25]{Fremlin89}).
Consider the function $\varphi_T^k:\algb^+ \to \algb_k$ defined by 
$$
	\varphi_T^k(b):=
	\begin{cases}
	b \otimes \dots \otimes b & \text{if $b\in T$},\\
	\0 & \text{otherwise}.
	\end{cases}
$$
Then $\varphi_T^k$ preserves disjointness, so there is a Boolean homomorphism
$\tilde{\varphi}_T^k:\alg\to \algb_k$ such that $\tilde{\varphi}_T^k(G_b)=\varphi_T^k(b)$ for all $b\in \algb^+$
(bear in mind that $\alg$ is isomorphic to the Boolean algebra freely generated by the $G_b$'s modulo the relations 
that $W(s,t) = \0$ if and only if $s\cap t \neq \emptyset$ or $\bigcap_{b\in s}b = \0$;
see Lemma~\ref{lem:zero} and the preceding comments). Now, it is not difficult to check that the formula 
\[
	\mu_T^k(a) := \lambda_k(\tilde{\varphi}_T^k(a))
\]
defines a probability measure 
$\mu_T^k$ on~$\alg$ satisfying the required property.
\end{proof}

From the technical point of view, the following subsets of~$S(K)$ will play a relevant role in our proof of
Theorem~\ref{theo:XXX}.

\begin{defi}\label{defi:SD}
Let $D$ be a finite partition of $\algb^+$. We denote by $S_D(K)$ (resp. $S'_D(K)$) the set of all $g\in S(K)$ which can be written as
$$
	g=\sum_{r \sub s}y_r 1_{\widehat{J(r)}}
$$
for some finite set $s\sub \algb^+$ such that $|T\cap s|\le 1$ for every $T\in D$
and some collection of real numbers (resp. nonzero real numbers) $\{y_r:r \sub s\}$.
\end{defi}

\begin{defi}
Let $D$ be a finite partition of $\algb^+$ and $C \sub D$. We define:
\begin{enumerate}
\item[(i)] $T_C:=\bigcup_{T\in C}T \sub \algb^+$;
\item[(ii)] a signed measure $\nu_C^k$ on~$\alg$ (for $k=1,2$) by
\[
	\nu_C^k:=\sum_{B\sub C} (-1)^{|C\sm B|}\mu_{T_B}^k;
\]
\item[(iii)] a function $\theta_C: C(K) \to \erre$ by
$$
	\theta_C(g):=\frac{ (\nu_C^1(g))^2}{\nu_C^2(g)}
$$ 
if $\nu_C^2(g)\neq 0$ and $\theta_C(g):=0$ otherwise;
\item[(iv)] a function $\eta_C: C(K) \to \erre$ by
$$
	\eta_C(g):=\frac{\nu_C^1(g)}{\theta_C(g)}
$$
if $\theta_C(g)\neq 0$ and $\eta_C(g):=0$ otherwise.
\end{enumerate}
\end{defi}

Clearly, the functions $\theta_C$ and $\eta_C$ defined above are $\Ba(C(K),w)$-measurable. The following
lemma collects several useful properties of such functions.

\begin{lem}\label{analysing:3}
Let $D$ be a finite partition of $\algb^+$ and let $g\in S_D(K)$ be as in Definition~\ref{defi:SD}. Fix $r \sub s$ and $k\in \{1,2\}$.
Writing
$$
	C_r:=\{T\in D:T\cap r \neq 0\},
$$ 
the following statements hold:
\begin{itemize}
\item[(i)] If $C \sub C_r$ then $\nu_C^k(J(r))=\mu^k_{T_C}(J(r))$.
\item[(ii)] If $C_r \subsetneq C \sub D$ then $\nu_C^k(J(r))=0$.
\item[(iii)] $\nu_{C_r}^k(g)=y_r (\lambda(\bigcap_{b\in r} b))^k$.
\item[(iv)] If $J(r)\neq \0$ then $\theta_{C_r}(g)=y_r$. 
\item[(v)] If $J(r)\neq \0$ and $y_r\neq 0$, then $\eta_{C_r}(g)=\lambda(\bigcap_{b\in r} b)$.
\end{itemize} 
\end{lem} 
\begin{proof} 
(i) For each $B \subsetneq C$ we have $r \not\subseteq T_B$ (because $C \sub C_r$) and so $\mu^k_{T_B}(J(r))=0$. 
Hence
$$
	\nu_C^k(J(r))=\sum_{B \sub C} (-1)^{|C\sm B|}\mu^k_{T_B}(J(r))=\mu^k_{T_C}(J(r)).
$$

(ii) For each $B \sub C$ with $C_r \not\subseteq B$ we have $r\not\subseteq T_B$ and so $\mu^k_{T_B}(J(r))=0$. 
On the other hand, given any $C_r \sub B \sub C$ we have $r \sub T_{C_r} \sub T_B$ and therefore
 $\mu^k_{T_B}(J(r))=(\lambda(\bigcap_{b\in r}b))^k$. Hence
\begin{multline*}
	\nu^k_C(J(r))=\sum_{B \sub C} (-1)^{|C\sm B|}\mu^k_{T_B}(J(r))
	=\sum_{C_r\sub B\sub C} (-1)^{|C\sm B|}\mu^k_{T_{B}}(J(r))= \\ =
	\left(\lambda\left(\bigcap_{b\in r}b\right)\right)^k\cdot 
	\sum_{C_r\sub B\sub C} (-1)^{|C\sm B|}.
\end{multline*}
The assumption $C \neq C_r$ implies
$$
	\sum_{C_r \sub B \sub C} (-1)^{|C\sm B|}=\sum_{A \sub C \sm C_r}(-1)^{|A|}=0,
$$
as can be easily checked by induction on $|C \sm C_r|$. Therefore, $\nu^k_C(J(r))=0$.

(iii) Take any $r' \sub s$. Note first that if $r' \not\subseteq r$ then $r' \not\subseteq T_B$ for every $B \sub C_r$, hence
$$
	\nu^k_{C_r}(J(r'))=\sum_{B \sub C_r}(-1)^{|C_r \sm B|}\mu^k_{T_B}(J(r'))=0.
$$
On the other hand, if $r' \subsetneq r$, then $C_{r'} \subsetneq C_r$ and (ii) implies
$\nu^k_{C_r}(J(r'))=0$. Observe also that $\nu_{C_r}^k(J(r))=(\lambda(\bigcap_{b\in r}b))^k$ (by~(i)).
It follows that
$$
	\nu_{C_r}^k(g)=
	\sum_{r' \sub s}y_{r'}\nu_{C_r}^k(J(r'))=
	y_{r}\nu_{C_r}^k(J(r))=y_r\left(\lambda\left(\bigcap_{b\in r}b\right)\right)^k.
$$

(iv) Bearing in mind~(iii) and the fact that $J(r)\neq \0$, we have $y_r=0$ if and only if $\nu_{C_r}^2(g)=0$. 
Thus, by the very definition of~$\theta_{C_r}$, the equality $\theta_{C_r}(g)=y_r$ holds whenever $y_r=0$. On the other hand,
if $y_r\neq 0$ then 
$$
	\theta_{C_r}(g)=
	\frac{(\nu^1_{C_r}(g))^2}{\nu^2_{C_r}(g)}\stackrel{{\rm (iii)}}{=}y_r
$$
and 
$$
	\eta_{C_r}(g)=\frac{\nu_{C_r}^1(g)}{\theta_{C_r}(g)}\stackrel{{\rm (iii)}}{=}\lambda\left(\bigcap_{b\in r}b\right),
$$
which proves~(v).
\end{proof}

Our next step is to prove that, for any $Z\sub \algb^+$ and $p\in \Nat$, 
the mapping $g\mapsto \mu_Z(g^p)$ is measurable with respect to the trace of $\Ba(C(K),w)$ on $S'(K)$
(see Lemma~\ref{lem:PuttingTogether} below). We begin by checking the measurability on subsets of the form~$S'_D(K)$.

\begin{lem} \label{analysing:4}
Let $Z\sub \algb^+$ be a set, $D$ a finite partition of $\algb^+$ finer than $\{Z,\algb^+\sm Z\}$ and $p\in \Nat$.
Then the mapping 
$$
	\phi_{Z,p}:S'_D(K) \to \erre,
	\quad \phi_{Z,p}(g):=\mu_Z(g^p),
$$
is measurable with respect to the trace of $\Ba(C(K),w)$ on $S'_D(K)$.
\end{lem}
\begin{proof}
Write $D(Z):=\{T \in D: T \sub Z\}$
and let $\Lambda$ be the set of all $\beta=(\beta_C)_{C \sub D(Z)}$ such that $\beta_C\in \Nat\cup \{0\}$ for all $C \sub D(Z)$
and $\sum_{C \sub D(Z)}\beta_C=p$. We write
$$
	\binom{p}{\beta}:=\frac{p!}{\prod_{C \sub D(Z)}\beta_C!}
	\quad
	\mbox{and}
	\quad
	C(\beta):=\bigcup_{\substack{C\sub D(Z)\\ \beta_C>0}}C. 
$$
To prove the measurability of $\phi_{Z,p}$ with respect to the trace of $\Ba(C(K),w)$ on $S'_D(K)$ it is sufficient to check 
that, for each $g \in S'_D(K)$, the following equality holds:
\begin{equation}\label{equation:formula}
	\phi_{Z,p}(g)=
	\sum_{\beta \in \Lambda}
	\binom{p}{\beta}
	\left(\prod_{\substack{C \sub D(Z)\\\beta_C>0}}\theta_C^{\beta_C}(g)\right)
	\eta_{C(\beta)}(g).
\end{equation}

\smallskip
{\sc Step~1.} Write
$$
	g=\sum_{r \sub s}y_r 1_{\widehat{J(r)}},
$$
where $s\sub \algb^+$ is finite, $|T\cap s|\le 1$ for every $T\in D$
and $y_r\in \erre \sm \{0\}$ for all $r\sub s$. Let $\Delta$ be the set of 
all $\delta=(\delta_r)_{r \sub s}$ such that $\delta_r \in \Nat\cup \{0\}$ for all $r \sub s$
and $\sum_{r \sub s}\delta_r=p$. Writing
$$
	\binom{p}{\delta}:=\frac{p!}{\prod_{r \sub s}\delta_r!}
	\quad
	\mbox{and}
	\quad
	r(\delta):=\bigcup_{\substack{r\sub s\\ \delta_r>0}}r,
$$
we have
$$
	g^p=
	\sum_{\delta \in \Delta} \binom{p}{\delta}
	\left( \prod_{\substack{r \sub s\\ \delta_r>0}} y_{r}^{\delta_r} \right)
	1_{\widehat{J(r(\delta))}}
$$
and so
\begin{equation}\label{equation:formulaMU}
	\phi_{Z,p}(g)=
	\mu_Z(g^p)=
	\sum_{\delta \in \Delta} \binom{p}{\delta}
	\left( \prod_{\substack{r \sub s \\ \delta_r>0}} y_{r}^{\delta_r} \right)
	\mu_Z\bigl(J(r(\delta))\bigr).
\end{equation}

\smallskip
{\sc Step~2.} Let $\delta \in \Delta$ such that $\mu_Z(J(r(\delta)))\neq 0$.
For any $r \sub s$ with $\delta_r>0$ we have $J(r) \supseteq J(r(\delta))$, hence $\mu_Z(J(r))\neq 0$ (in particular, $J(r)\neq \0$)
and so $r \sub Z$, which implies that $C_r=\{T\in D:T \cap r\neq \emptyset\} \sub D(Z)$.
Set $\beta=(\beta_C)_{C \sub D(Z)}$ by declaring
\begin{equation}\label{equation:defineBETA}
	\beta_C:=
	\begin{cases}
		\delta_r & \text{if $C=C_r$ for some $r\sub s$ with $\delta_r>0$,}\\
		0 & \text{otherwise}.
	\end{cases}
\end{equation}
Then $\sum_{C \sub D(Z)}\beta_C=\sum_{r \sub s}\delta_r=p$, so that $\beta\in \Lambda$. Moreover, we have
$$
	C(\beta)=\bigcup_{\substack{C \sub D(Z) \\ \beta_C>0}}C=
	\bigcup_{\substack{r \sub s\\ \delta_r>0}}C_r=C_{r(\delta)}.
$$
Since $\mu_Z(J(r(\delta)))\neq 0$, we have $J(r(\delta))\neq \0$ and $\mu_Z(J(r(\delta)))=\lambda(\bigcap_{b\in r(\delta)}b)$.
Thus, from Lemma~\ref{analysing:3}(v) it follows that $\eta_{C(\beta)}(g)=\eta_{C_{r(\delta)}}(g)=\mu_Z(J(r(\delta)))$.
On the other hand, for each $r \sub s$ with $\delta_r>0$ we have $y_r=\theta_{C_r}(g)$ by Lemma~\ref{analysing:3}(iv), hence
$$
	\prod_{\substack{C \sub D(Z)\\\beta_C>0}}\theta_C^{\beta_C}(g)=
	\prod_{\substack{r \sub s\\ \delta_r>0}}\theta_{C_r}^{\delta_r}(g)=
	\prod_{\substack{r \sub s\\ \delta_r>0}}y_r^{\delta_r}.
$$
Therefore
$$	
	\binom{p}{\beta}
	\left(\prod_{\substack{C \sub D(Z)\\\beta_C>0}}\theta_C^{\beta_C}(g)\right)
	\eta_{C(\beta)}(g)=
	\binom{p}{\delta}
	\left(\prod_{\substack{r \sub s\\ \delta_r>0}}y_r^{\delta_r}\right)
	\mu_Z\bigl(J(r(\delta))\bigr).
$$
This shows that each nonzero summand of~\eqref{equation:formulaMU} can be written
as a summand of~\eqref{equation:formula}. Note also that 
if $\delta' \in \Delta$ satisfies $\mu_Z(J(r(\delta')))\neq 0$ and we 
define $\beta'=(\beta'_C)_{C \sub D(Z)}\in \Lambda$ as in~\eqref{equation:defineBETA}
(with $\delta$ replaced by~$\delta'$), then $\beta\neq \beta'$ whenever $\delta\neq \delta'$.

\smallskip
{\sc Step~3.} Let $\beta \in \Lambda$ such that
\begin{equation}\label{equation:rhs}
	\left(\prod_{\substack{C \sub D(Z)\\\beta_C>0}}\theta_C^{\beta_C}(g)\right)
	\eta_{C(\beta)}(g)\neq 0.
\end{equation}

Fix $C \sub D(Z)$ with $\beta_C>0$. We claim that $r_C:=T_C \cap s$ satisfies
$C_{r_C}=C$. Indeed, the inclusion $C_{r_C} \sub C$ is clear. To prove the reverse inclusion, we argue by contradiction.
Suppose $C_{r_C}\subsetneq C$. By Lemma~\ref{analysing:3}(ii), we have $\nu_C^1(J(r'))=0$ whenever $r' \sub T_C$. Since
we also have $\nu_C^1(J(r'))=0$ for every $r' \sub s$ with $r' \not\subseteq T_C$ (by the very definition of~$\nu_C^1$), it follows
that $\nu_C^1(g)=\sum_{r' \sub s} y_{r'}\nu_C^1(J(r'))=0$, hence $\theta_C(g)=0$, which contradicts~\eqref{equation:rhs}.
This shows that $C_{r_C}=C$, as claimed. Now Lemma~\ref{analysing:3}(iii) ensures that
$$
	\nu_C^1(g)=\nu_{C_{r_C}}^1(g)=y_{r_C}\lambda\left(\bigcap_{b\in r_C}b \right).
$$
Since $\nu_C^1(g)\neq 0$, the previous equality implies that $J(r_C)\neq \0$. From Lemma~\ref{analysing:3}(iv)
it follows that $y_{r_C}=\theta_{C_{r_C}}(g)=\theta_C(g)$.

Set $\delta=(\delta_r)_{r\sub s}$ by declaring
\begin{equation}\label{equation:defineDELTA}
	\delta_r:=
	\begin{cases}
	\beta_{C} & \text{if $r=r_C$ for some $C \sub D(Z)$ with $\beta_C>0$},\\
	0 & \text{otherwise}.
	\end{cases}
\end{equation}
Then $\sum_{r\sub s}\delta_r=\sum_{C \sub D(Z)}\beta_C=p$, hence $\delta \in \Delta$. From our previous considerations
we deduce that
$$
	\prod_{\substack{C \sub D(Z)\\\beta_C>0}} \theta_C^{\beta_C}(g)=
	\prod_{\substack{C \sub D(Z)\\\beta_C>0}} y_{r_C}^{\beta_C}=
	\prod_{\substack{r \sub s \\ \delta_r>0}} y_r^{\delta_r}.
$$

Moreover, we claim that $\eta_{C(\beta)}(g)=\mu_Z(J(r(\delta)))$. Indeed, since
$$
	C(\beta)=\bigcup_{\substack{C \sub D(Z) \\ \beta_C>0}}C=
	\bigcup_{\substack{C \sub D(Z) \\ \beta_C>0}}C_{r_C}=
	\bigcup_{\substack{r \sub s\\ \delta_r>0}}C_r=C_{r(\delta)},
$$
we have $\eta_{C(\beta)}(g)=\eta_{C_{r(\delta)}}(g)$ and so~\eqref{equation:rhs}
implies that $\nu_{C_{r(\delta)}}^1(g)\neq 0$. Bearing in mind
Lemma~\ref{analysing:3}(iii) and the fact that $y_{r(\delta)}\neq 0$, we infer that $J(r(\delta))\neq \0$.
An appeal to Lemma~\ref{analysing:3}(v) now yields $\eta_{C(\beta)}(g)=\lambda(\bigcap_{b\in r(\delta)}b)$. On the other hand,
the fact that 
$$
	r(\delta)=\bigcup_{\substack{r \sub s\\ \delta_r>0}} r=\bigcup_{\substack{C \sub D(Z)\\ \beta_C>0}} r_C=
	\bigcup_{\substack{C \sub D(Z)\\ \beta_C>0}} T_C\cap s 	\sub Z
$$
ensures that $\mu_Z(J(r(\delta)))=\lambda(\bigcap_{b\in r(\delta)}b)$. It follows that
$\eta_{C(\beta)}(g)=\mu_Z(J(r(\delta)))$. 

Therefore 
$$
	\binom{p}{\beta}\left(\prod_{\substack{C \sub D(Z)\\\beta_C>0}}\theta_C^{\beta_C}(g)\right)
	\eta_{C(\beta)}(g)=
	\binom{p}{\delta}\left(\prod_{\substack{r \sub s \\ \delta_r>0}} y_r^{\delta_r}\right)
	\mu_Z(J(r(\delta))).
$$	
This shows that each nonzero summand of~\eqref{equation:formula} can be written
as a summand of~\eqref{equation:formulaMU}. Note 
that if $\beta' \in \Lambda$ satisfies~\eqref{equation:rhs} 
(with $\beta$ replaced by~$\beta'$) and we define
$\delta'=(\delta'_r)_{r\sub s}\in \Delta$ as in~\eqref{equation:defineDELTA}
(with $\beta$ replaced by~$\beta'$),
then $\delta \neq \delta'$ whenever $\beta \neq \beta'$.

Thus, equality \eqref{equation:formula} holds true and the
proof is over.
\end{proof}

The following folklore fact will allow us to prove Lemma~\ref{lem:PuttingTogether} as an easy consequence
of Lemma~\ref{analysing:4} above.

\begin{remark}\label{analysing:5}
Let $(\Omega,\Sigma)$ be a measurable space. Write $\Omega=\bigcup_{j\in \N}\Omega_j$ where $\Omega_j \sub \Omega_{j+1}$
for all $j\in \N$. Let $A \sub \Omega$ be a set such that, 
for each $j\in \Nat$, the intersection $A\cap \Omega_j$ belongs to the trace of~$\Sigma$ on~$\Omega_j$. 
Then $A\in \Sigma$.
\end{remark}
\begin{proof}
For each $j\in \N$ we have $A\cap \Omega_j= E_j\cap \Omega_j$ for some $E_j\in\Sigma$. We claim that $A=\bigcup_{k\in \N}\bigcap_{j\geq k}E_j$.
Indeed, we have
$$
	A=\bigcup_{k\in \N}A\cap \Omega_k=\bigcup_{k\in \N}\bigcap_{j\geq k}A \cap \Omega_j=
	\bigcup_{k\in \N}\bigcap_{j\geq k}E_j \cap \Omega_j \sub \bigcup_{k\in \N}\bigcap_{j\geq k}E_j.
$$
On the other hand, for each $k\in \N$, we have
$$
	\bigcap_{j\geq k}E_j=
	\bigcup_{n\geq k}\bigcap_{j\geq k}E_j \cap \Omega_n \sub
	\bigcup_{n\geq k}E_n \cap \Omega_n=\bigcup_{n\geq k}A \cap \Omega_n=A.
$$
It follows that $A=\bigcup_{k\in \N}\bigcap_{j\geq k}E_j\in \Sigma$.
\end{proof}

\begin{lem}\label{lem:PuttingTogether}
Let $Z\sub \algb^+$ be a set and $p\in \Nat$.
Then the mapping 
$$
	\phi_{Z,p}:S'(K) \to \erre,
	\quad \phi_{Z,p}(g):=\mu_Z(g^p),
$$
is measurable with respect to the trace of $\Ba(C(K),w)$ on $S'(K)$.
\end{lem}
\begin{proof}
Since $|\algb^+|=\mathfrak{c}$, there is a sequence $D(1),D(2),\dots$ of finite partitions of $\algb^+$, each one being finer than
$\{Z,\algb^+\sm Z\}$, such that:
\begin{itemize}
\item $D(j+1)$ is finer than $D(j)$ for all $j\in \N$;
\item for every $s \sub \algb^+$ finite there is $j\in \N$ such that $|T\cap s|\leq 1$ for all $T\in D(j)$.
\end{itemize}
Indeed, let $\xi: \{0,1\}^{\N} \to \algb^+$ be any bijection and, for each $j\in \N$ and $\sigma\in \{0,1\}^j$, set
$$
	E^\sigma_j:=\{x\in \{0,1\}^{\N}: \, x(i)=\sigma(i) \mbox{ for every }i=1,\dots,j\}.
$$
Then the partitions 
$$
	D(j):=\bigl\{\xi(E^\sigma_j)\cap Z: \sigma\in \{0,1\}^j\bigr\}
	\cup
	\bigl\{\xi(E^\sigma_j)\sm Z: \sigma\in \{0,1\}^j\bigr\}, \quad j\in\N,
$$
fulfill the required properties. 

Clearly, $S'(K)=\bigcup_{j\in \N} S'_{D(j)}(K)$ and $S'_{D(j)}(K) \sub S'_{D(j+1)}(K)$ for all $j\in \N$.
The measurability of $\phi_{Z,p}$ with respect to the trace of $\Ba(C(K),w)$ on~$S'(K)$
now follows from Lemma~\ref{analysing:4} and Remark~\ref{analysing:5}.
\end{proof}

We have already gathered all the tools needed to prove the main result of this subsection:

\begin{theo}\label{theo:XXX}
The restriction of the supremum norm to~$S'(K)$ is measurable with respect to the trace of $\Ba(C(K),w)$ on~$S'(K)$.
\end{theo}
\begin{proof}
We fix a countable algebra~$\cZ$ on~$\algb^+$ which separates the points of~$\algb^+$ (the algebra
of clopen subsets of $\{0,1\}^\Nat$ can be transferred to $\algb^+$ via any bijection
between $\{0,1\}^\Nat$ and~$\algb^+$). We claim that 
\begin{equation}\label{equation:formulaNORM}
	\|g\| =
	\sup_{Z\in\cZ} \limsup_{p\to\infty} \left( \mu_Z(g^{2p})\right)^{\frac{1}{2p}}
	\quad
	\mbox{for every }g\in C(K).
\end{equation}
Indeed, the inequality ``$\ge$'' is obvious (each $\mu_Z$ is a probability measure).
To verify the reverse inequality, fix $g\in C(K)$ and take $\eps>0$. By Lemma~\ref{lem:density}(i)
there exist finite disjoint sets $r,s \sub \algb^+$ such that $|g(\cF)|\ge \|g\|-\eps$ for every $\cF\in \widehat{W(r,s)}\neq \emptyset$. 
Since $\cZ$ separates the points of~$\algb^+$, we can find 
$Z\in\cZ$ such that $r \sub Z$ and $s \cap Z=\emptyset$, hence $\mu_Z(W(r,s))=\lambda(\bigcap_{b\in r} b)>0$.
Since
$$
	\left(\mu_Z(g^{2p})\right)^{\frac{1}{2p}}
	\ge (\|g\|-\eps)\bigl(\mu_Z(W(r,s))\bigr)^{\frac{1}{2p}} \quad \mbox{for every }p\in \N,
$$
we have
$$
	\limsup_{p\to\infty} \left( \mu_Z(g^{2p})\right)^{\frac{1}{2p}} \geq
	(\|g\|-\eps)\lim_{p\to \infty}\bigl(\mu_Z(W(r,s))\bigr)^{\frac{1}{2p}}=\|g\|-\eps.
$$
As $\eps>0$ is arbitrary, equality~\eqref{equation:formulaNORM} holds true.

Once we know that $\|\cdot\|$ is expressed by the formula~\eqref{equation:formulaNORM}, 
the assertion follows from Lemma~\ref{lem:PuttingTogether}.
\end{proof}

\subsection{Measurability on the set of simple functions}\label{subsection:Kunen}

Any element of~$S(K)$ admits a representation which cannot be simplified in a sense, as the following lemma shows.

\begin{lem}\label{lem:simplification}
Let $g\in S(K)$. Then there exist a finite set $s\sub \algb^+$ and a collection of real numbers $\{z_r:r \sub s\}$
such that:
\begin{enumerate}
\item[(i)] $g=\sum_{r \sub s}z_{r}1_{\widehat{W(r,s\setminus r)}}$;
\item[(ii)] there is no $s' \subsetneq s$ such that 
$$
	z_r=z_{r'} \quad \mbox{whenever} \quad r \cap s'=r' \cap s', \ \bigcap_{b\in r}b \neq \0
	\mbox{ and }\bigcap_{b\in r'}b \neq \0.
$$
\end{enumerate}
\end{lem}
\begin{proof}
Of course, we can write $g$ as in~(i). To get a representation satisfying~(ii), we proceed by induction on~$|s|$. The case $s=\emptyset$ being obvious, we assume
that $s\neq \emptyset$ and that the induction hypothesis holds. Assume that (ii) fails and fix 
$s' \subsetneq s$ such that 
$$
	z_r=z_{r'} \quad \mbox{whenever} \quad r \cap s'=r' \cap s', \ \bigcap_{b\in r}b \neq \0
	\mbox{ and }\bigcap_{b\in r'}b \neq \0.
$$
For any $t \sub s'$ with $\bigcap_{b\in t}b \neq \0$, let $\cA_t$ be the collection of all
$r \sub s$ such that $r \cap s'=t$ and $\bigcap_{b\in r}b\neq \0$. Then 
$z_r=z_t$ for every $r \in \cA_t$ and
$$
	W(t,s'\sm t)=\bigcup_{r\in \cA_t} W(r,s\sm r),
$$
as can be easily checked. Hence
\begin{multline*}
	g=
	\sum_{r \sub s}z_{r}1_{\widehat{W(r,s\setminus r)}}=
	\sum_{t \sub s'}
	\sum_{\substack{r \sub s \\ r \cap s'=t}}z_r 1_{\widehat{W(r,s\setminus r)}}=
	\sum_{\substack{t \sub s' \\ \bigcap_{b\in t}b \neq \0}}
	\sum_{r \in \cA_t}z_r 1_{\widehat{W(r,s\setminus r)}}= \\ =
	\sum_{\substack{t \sub s' \\ \bigcap_{b\in t}b \neq \0}}
	\sum_{r \in \cA_t}z_t 1_{\widehat{W(r,s\setminus r)}}=
	\sum_{\substack{t \sub s' \\ \bigcap_{b\in t}b \neq \0}}
	z_t 1_{\widehat{W(t,s'\setminus t)}}=
	\sum_{t \sub s'} z_t 1_{\widehat{W(t,s'\setminus t)}}.
\end{multline*}	 
Since $|s'|<|s|$, the induction hypothesis now ensures that $g$ admits a representation satisfying both (i) and~(ii).
\end{proof}

\begin{defi}\label{defi:AD}
Let $D$ be a finite partition of~$\algb^+$. We denote by $A_D(K)$ the set of all $g\in S(K)$ which can be written as
$$
	g=\sum_{r \sub s}z_{r}1_{\widehat{W(r,s\setminus r)}}
$$
for some finite set $s\sub \algb^+$ and some collection of real numbers $\{z_r: r \sub s\}$ such that:
\begin{enumerate}
\item[(i)] $|T\cap s|=1$ for every $T\in D$;
\item[(ii)] there is no $s' \subsetneq s$ such that 
$$
	z_r=z_{r'} \quad \mbox{whenever} \quad r \cap s'=r' \cap s', \ \bigcap_{b\in r}b \neq \0
	\mbox{ and }\bigcap_{b\in r'}b \neq \0.
$$
\end{enumerate}
\end{defi}

Our next step is to prove that the sets $A_D(K)$ defined above belong to the trace of~$\Ba(C(K),w)$ on~$S(K)$
(see Corollary~\ref{cor:SDmeasurable} below). From now on we fix a {\em countable} algebra~$\cZ$ on~$\algb^+$ which separates the points of~$\algb^+$ (like in the 
proof of Theorem~\ref{theo:XXX}).
 
\begin{lem}\label{lem:Condition2}
Let $D$ be a finite partition of~$\algb^+$ with $D \sub \cZ$ and let $g\in A_D(K)$. Then
for each $T_0\in D$ there is $T\in \cZ$ such that $\mu_{T\sm T_0}(g)\neq\mu_{T\cup T_0}(g)$.
\end{lem}
\begin{proof}
Our proof is by contradiction. Suppose there is $T_0\in D$ such that
\begin{equation}\label{equation:contrad}
	\mu_{T\sm T_0}(g)=\mu_{T\cup T_0}(g)
	\quad
	\mbox{for every }T\in \cZ.
\end{equation} 
Write $g=\sum_{r \sub s}z_{r}1_{\widehat{W(r,s\setminus r)}}$
as in Definition~\ref{defi:AD}.
Let $s':=s \sm T_0 \subsetneq s$. In order to reach a contradiction, we claim that
$$
	z_r=z_{r'} \quad \mbox{whenever} \quad r \cap s'=r' \cap s', \ \bigcap_{b\in r}b \neq \0
	\mbox{ and }\bigcap_{b\in r'}b \neq \0.
$$
Indeed, assume that $r\neq r'$ and proceed by induction on $|r \cap s'|=|r' \cap s'|$. 

Suppose first that $r \cap s'=r' \cap s'=\emptyset$.
Then we have $r=\emptyset$ and $r'=T_0 \cap s$ (or vice versa). By~\eqref{equation:contrad} we have
$\mu_{\emptyset}(g)=\mu_{T_0}(g)$. Note that 
$\mu_{\emptyset}(g)=z_{\emptyset}$ and, writing $T_0 \cap s=\{b_0\}$, we have 
$\mu_{T_0}(g)=z_{\emptyset}(1-\lambda(b_0))+z_{T_0\cap s}\lambda(b_0)$.
It follows that 
$$
	z_{\emptyset}=z_{\emptyset}(1-\lambda(b_0))+z_{T_0\cap s}\lambda(b_0),
$$ 
hence $z_\emptyset=z_{T_0\cap s}$, as required.

Suppose now that $r \cap s'=r'\cap s'\neq \emptyset$, together with $\bigcap_{b\in r}b\neq \0 \neq \bigcap_{b\in r'}b$, 
and the inductive hypothesis. 
Since $r\neq r'$, we have either $b_0\in r$ and $b_0\not \in r'$ or vice versa. We assume for instance that $b_0\in r$ and $b_0\not\in r'$. Then
$$
	r=\{b_0\}\cup(r\cap s') \quad \mbox{and} \quad
	r'=r\cap s'.
$$
By the inductive hypothesis, 
\begin{equation}\label{equation:ind}
	z_{r_0}=z_{r_0\cup \{b_0\}}
	\quad
	\mbox{for every }r_0\subsetneq r'.
\end{equation}
Set 
$$
	T_1:=T_0\cup \bigcup\{T\in D: \, T\cap s \sub r'\} \in \cZ
$$
and observe that $T_1\cap s=r$. Writing 
$$
	w(t,t'):=\bigcap_{b\in t}b\sm \bigcup_{b'\in t'}b'\in \algb
$$ 
for any pair of finite sets $t,t' \sub \algb^+$, we have
\begin{multline}\label{multline:muT1}
	\mu_{T_1}(g)=
	\sum_{r_0 \sub r}z_{r_0} \lambda\bigl(w(r_0,r\sm r_0)\bigr)=
	\sum_{r_0 \sub r'}z_{r_0} \lambda\bigl(w(r_0,r\sm r_0)\bigr)+ 
	\sum_{\substack{r_0\sub r \\ b_0\in r_0}}z_{r_0} \lambda\bigl(w(r_0,r'\sm r_0)\bigr)= \\ =
	z_{r'} \lambda\bigl(w(r',\{b_0\})\bigr)+
	\sum_{r_0 \subsetneq r'}z_{r_0} \lambda\bigl(w(r_0,r\sm r_0)\bigr) +\\
	+
	z_{r} \lambda\bigl(w(r,\emptyset)\bigr)+
	\sum_{r_0 \subsetneq r'}z_{r_0\cup\{b_0\}} \lambda\bigl(w(r_0\cup \{b_0\},r' \sm r_0) \bigr).
\end{multline}
For each $r_0 \subsetneq r'$, the elements
$w(r_0,r\sm r_0)$ and $w(r_0\cup \{b_0\},r'\sm r_0)$
are disjoint and their union is $w(r_0,r'\sm r_0)$, hence \eqref{equation:ind} yields
$$
	z_{r_0} \lambda\bigl(w(r_0,r'\sm r_0)\bigr)=
	z_{r_0} \lambda\bigl(w(r_0,r\sm r_0)\bigr)+
	z_{r_0\cup\{b_0\}} \lambda\bigl(w(r_0\cup \{b_0\},r' \sm r_0) \bigr).	
$$
From~\eqref{multline:muT1} it follows that
\begin{equation}\label{multline:after}
	\mu_{T_1}(g)=
	z_{r'} \lambda\bigl(w(r',\{b_0\})\bigr)+
	z_{r} \lambda\bigl(w(r,\emptyset)\bigr)
	+
	\sum_{r_0 \subsetneq r'}z_{r_0} \lambda\bigl(w(r_0,r'\sm r_0)\bigr).
\end{equation}
Bearing in mind that $(T_1 \sm T_0)\cap s=r'$, we also have
\begin{equation}\label{multline:after2}
	\mu_{T_1\sm T_0}(g)=
	\sum_{r_0 \sub r'}z_{r_0} \lambda\bigl(w(r_0,r'\sm r_0)\bigr)
	=z_{r'}\lambda\bigl(w(r',\emptyset)\bigr)+
	\sum_{r_0 \subsetneq r'}
	z_{r_0} \lambda\bigl(w(r_0,r'\sm r_0)\bigr).
\end{equation}
Since $\mu_{T_1\sm T_0}(g)=\mu_{T_1}(g)$ (by~\eqref{equation:contrad}), equalities
\eqref{multline:after} and~\eqref{multline:after2} yield
$$
	z_{r'} \lambda\bigl(w(r',\{b_0\})\bigr)+
	z_{r} \lambda\bigl(w(r,\emptyset)\bigr)=
	z_{r'}\lambda\bigl(w(r',\emptyset)\bigr),
$$
therefore $z_{r} \lambda(\bigcap_{b\in r}b)=z_{r'}\lambda(\bigcap_{b\in r}b)$
and so $z_{r}=z_{r'}$. This finishes the proof.
\end{proof}

\begin{remark}\label{remark:EqualMeasures}
\rm Let $g\in S(K)$ be written as $g=\sum_{r \sub s}z_{r}1_{\widehat{W(r,s\setminus r)}}$
for some finite set $s\sub \algb^+$ and $z_r\in \erre$.
If $T,T' \sub \algb^+$ satisfy $T\cap s=T'\cap s$, then $\mu_T(g)=\mu_{T'}(g)$.
\end{remark}
\begin{proof}
For every $r \sub s$ we have $r \sub T$ if and only if $r \sub T'$. In this case,
$$
	\mu_T(W(r,s))=\lambda\left(\bigcap_{b\in r}b \sm \bigcap_{b\in s\cap T}b\right)=\mu_{T'}(W(r,s)).
$$
Hence $\mu_T(g)=\sum_{r \sub T}z_r\mu_T(W(r,s))=\sum_{r \sub T'}z_r\mu_{T'}(W(r,s))=\mu_{T'}(g)$.
\end{proof}

\begin{lem}\label{lem:SDmeasurable}
Let $D$ be a finite partition of~$\algb^+$ with $D \sub \cZ$ and let $g\in S(K)$. Then $g\in A_D(K)$ if and only if the following two statements hold:
\begin{enumerate}
\item[($\star$)] for each $T_0\in D$ there is $T\in \cZ$ such that $\mu_{T\sm T_0}(g)\neq\mu_{T\cup T_0}(g)$;
\item[($\star\star$)] for each $T_0\in D$ and each finite partition $D_0$ of~$T_0$ with~$D_0 \sub\cZ$, there is
$Z_0\in D_0$ such that $\mu_{T\sm Z}(g)=\mu_{T\cup Z}(g)$ for every $Z\in D_0\sm \{Z_0\}$ and $T\in \cZ$.
\end{enumerate}
\end{lem}
\begin{proof} 
{\sc ``Only if'' part.} Suppose $g\in A_D(K)$ and write $g=\sum_{r \sub s}z_{r}1_{\widehat{W(r,s\setminus r)}}$ 
as in Definition~\ref{defi:AD}. Statement~($\star$) holds by Lemma~\ref{lem:Condition2}. To check~($\star\star$),
take $T_0\in D$ and fix a finite partition $D_0$ of~$T_0$ with~$D_0 \sub\cZ$. 
Since $T_0\cap s$ is a singleton,
there is $Z_0\in D_0$ such that $Z\cap s=\emptyset$ for every $Z\in D_0\sm \{Z_0\}$, and so for any $T\in \cZ$ we have
$(T\sm Z)\cap s=(T \cup Z) \cap s$, hence $\mu_{T \sm Z}(g)=\mu_{T\cup Z}(g)$ (Remark~\ref{remark:EqualMeasures}).

{\sc ``If'' part.} Write $g=\sum_{r \sub s}z_{r}1_{\widehat{W(r,s\setminus r)}}$
for some finite set $s\sub \algb^+$ and some collection of real numbers $\{z_r:r \sub s\}$.
By Lemma~\ref{lem:simplification}, this representation can be chosen in such a way that there is no $s' \subsetneq s$ such that 
$$
	z_r=z_{r'} \quad \mbox{whenever} \quad r \cap s'=r' \cap s', \ \bigcap_{b\in r}b \neq \0
	\mbox{ and }\bigcap_{b\in r'}b \neq \0.
$$ 
In order to prove that $g\in A_D(K)$ we only have to check that $|T_0\cap s|= 1$ for every $T_0\in D$. Observe first
that, for each $T_0\in D$, condition~($\star$) tells us that there is $T \in \cZ$ such that $\mu_{T\sm T_0}(g)\neq \mu_{T \cup T_0}(g)$, 
hence $(T \sm T_0)\cap s \neq (T\cup T_0)\cap s$ (Remark~\ref{remark:EqualMeasures}) and so $T_0 \cap s\neq \emptyset$.
Thus, we can find a finite partition $D' \sub \cZ$ of~$\algb^+$ finer than~$D$ such that $|T'\cap s|=1$ for every $T'\in D'$.
Therefore, $g\in A_{D'}(K)$.
 
Fix $T_0\in D$ and set $D_0:=\{T'\in D':T' \sub T_0\}$. 
By Lemma~\ref{lem:Condition2} applied to $g$ and~$D'$, for each $T'\in D_0$ there is $T''\in \cZ$ such that 
$\mu_{T''\sm T'}(g)\neq \mu_{T''\cup T'}(g)$. This fact and condition~($\star\star$) 
yield $|D_0|=1$, that is, $D_0=\{T_0\}$ and so $|T_0\cap s|=1$. As 
$T_0\in D$ is arbitrary, $g\in A_D(K)$ and the proof is over.
\end{proof}

\begin{cor}\label{cor:SDmeasurable}
Let $D$ be a finite partition of~$\algb^+$. Then $A_D(K)$ belongs to the trace of~$\Ba(C(K),w)$ on~$S(K)$.
\end{cor}
\begin{proof}
We can assume without loss of generality (by enlarging~$\cZ$ if necessary) that $D \sub \cZ$. 
Since~$\cZ$ is countable, Lemma~\ref{lem:SDmeasurable} gives the result.
\end{proof}

Our next task is to prove that, under the assumption that~$\mathfrak{c}$ is a Kunen cardinal, 
the restriction of the supremum norm to any set of the form~$A_D(K)$ is 
relatively $\Ba(C(K),w)$-measurable (Lemma~\ref{lem:NormRestriction}).

\begin{lem}\label{lem:Phi}
Let $\Phi:\algb^+ \to 2^{\cZ}$ be the mapping defined by $\Phi(b):=(1_T(b))_{T\in \cZ}$. Set $\Omega:=\Phi(\algb^+)$ and let $\Sigma$
be the trace of ${\rm Borel}(2^{\cZ})$ on~$\Omega$. Then for each $n\in \Nat$ the mapping
$$
	\Phi^n: ((\algb^+)^n,\otimes_n\sigma(\cZ)) \to (\Omega^n,\otimes_n\Sigma),
	\quad
	\Phi^n(b_1,\dots,b_n):=(\Phi(b_1),\dots,\Phi(b_n)),
$$
is an isomorphism of measurable spaces.
\end{lem}
\begin{proof}
It suffices to prove the case $n=1$. Clearly, $\Phi$ is one-to-one (because $\cZ$ separates the points of~$\algb^+$) and
$\sigma(\cZ)$-$\Sigma$-measurable. On the other hand, for each $T_0\in \cZ$ we have
$$
	\Phi(T_0)=\{(1_T(b))_{T\in \cZ}: \, b\in T_0\}=\Omega \cap \{(x_T)_{T\in \cZ}\in 2^{\cZ}: \, x_{T_0}=1\}\in \Sigma,
$$
hence $\Phi^{-1}$ is $\Sigma$-$\sigma(\cZ)$-measurable.
\end{proof}

\begin{lem}\label{lem:Function}
Let $D$ be a finite partition of~$\algb^+$ with $D\sub \cZ$, let $T_0\in D$ and $T\in \cZ$. Let $g\in A_D(K)$. Write
$g=\sum_{r \sub s}z_{r}1_{\widehat{W(r,s\setminus r)}}$ as in Definition~\ref{defi:AD} and $T_0\cap s=\{b_0\}$.
\begin{enumerate}
\item The following statements are equivalent:
\begin{enumerate}
\item[(i)] $b_0\in T$;
\item[(ii)] $\mu_{\bar{T}}(g)=\mu_{\bar{T}\cup T_0}(g)$ for every $\bar{T}\in \cZ$ such that $\bar{T}\cap T_0=T\cap T_0$.
\end{enumerate}
\item The following statements are equivalent:
\begin{enumerate}
\item[(i')] $b_0 \not\in T$;
\item[(ii')] $\mu_{\bar{T}}(g)=\mu_{\bar{T}\sm T_0}(g)$ for every $\bar{T}\in \cZ$ such that $\bar{T}\cap T_0=T\cap T_0$.
\end{enumerate}
\end{enumerate}
\end{lem}
\begin{proof} 
(i)$\impli$(ii) Let $\bar{T}\in \cZ$ be such that $\bar{T}\cap T_0=T\cap T_0$.
Since $b_0\in T$ by assumption, we have $b_0\in T\cap T_0 \sub \bar{T}$ and so
$(\bar{T}\cup T_0)\cap s=\bar{T}\cap s$. Bearing in mind Remark~\ref{remark:EqualMeasures}, we get
$\mu_{\bar{T}}(g)=\mu_{\bar{T}\cup T_0}(g)$. A similar argument yields (i')$\impli$(ii').

Now, in order to prove (ii)$\impli$(i) and (ii')$\impli$(i'), it is enough to check that statements (ii) and (ii') cannot hold
simultaneously. To this end, pick $T^*\in \cZ$ such that $\mu_{T^*\sm T_0}(g)\neq\mu_{T^*\cup T_0}(g)$
(we apply Lemma~\ref{lem:Condition2}) and set 
$$
	\bar{T}:=(T^*\sm T_0)\cup (T\cap T_0)\in \cZ.
$$ 
Clearly, $\bar{T}\cup T_0=T^*\cup T_0$ and $\bar{T}\sm T_0=T^*\sm T_0$, 
hence we have either $\mu_{\bar{T}\cup T_0}(g)\neq \mu_{\bar{T}}(g)$ or $\mu_{\bar{T}\sm T_0}(g)\neq \mu_{\bar{T}}(g)$.
Since $\bar{T}\cap T_0=T\cap T_0$, this shows that
either (ii) or (ii') fails.
\end{proof}

\begin{remark}\label{remark:composition}
Let $D=\{T_1,\dots,T_n\}$ be a finite partition of~$\algb^+$ with $D\sub \cZ$.
\begin{enumerate}
\item Let $i\in \{1,\dots,n\}$. Given $T\in \cZ$, since statements (ii) and (ii') in 
Lemma~\ref{lem:Function} (applied to~$T_i$) are independent of the representation of~$g\in A_D(K)$, there is 
a mapping $\psi_{D,T_i,T}: A_D(K) \to \{0,1\}$ 
such that, for any $g=\sum_{r \sub s}z_{r}1_{\widehat{W(r,s\setminus r)}}$ as in Definition~\ref{defi:AD}
and writing $T_i\cap s=\{b_i\}$, we have 
$$
	\psi_{D,T_i,T}(g):=
	\begin{cases}
	1 & \text{if $b_i \in T$},\\
	0 & \text{if $b_i\not\in T$}.
	\end{cases}
$$
$\psi_{D,T_i,T}$ is measurable with respect to the trace
of $\Ba(C(K),w)$ on~$A_D(K)$, thanks to Lemma~\ref{lem:Function} (applied to~$T_i$).
Define $\psi_{D,T_i}: A_D(K) \to 2^{\cZ}$ by $\psi_{D,T_i}(g):=(\psi_{D,T_i,T}(g))_{T\in \cZ}$. 
Observe that 
$$
	\psi_{D,T_i}(A_D(K)) \sub \Omega,
$$
because for any $g\in A_D(K)$ as above we have $\psi_{D,T_i}(g)=\Phi(b_i)$.

\item Thus, we can consider the mapping
$$
	\psi_D:A_D(K) \to \Omega^n, \quad
	\psi_D(g):=(\psi_{D,T_1}(g),\dots,\psi_{D,T_n}(g)).
$$
Clearly, $\psi_D$ is measurable with respect to $\otimes_n \Sigma$ and the trace
of $\Ba(C(K),w)$ on~$A_D(K)$.

\item Let $P \sub \{1,\dots,n\}$. Define 
$\zeta_{n,P}: (\algb^+)^n \to \erre$ by
$$
	\zeta_{n,P}(b'_1,\dots,b'_n):=\lambda\left(\bigcap_{i\in P}b'_i\right).
$$
Then the mapping
$L_{D,P}: A_D(K) \to \erre$ given by $L_{D,P}:=\zeta_{n,P}\circ(\Phi^n)^{-1}\circ\psi_D$ satisfies
$L_{D,P}(g)=\lambda(\bigcap_{i\in P}b_i)$ for every $g\in A_D(K)$ as above.
\end{enumerate}
\end{remark}

From now on we deal with the additional assumption that $\mathfrak{c}$ is a Kunen cardinal.

\begin{lem}\label{lem:zetaP}
Suppose $\mathfrak{c}$ is a Kunen cardinal. Then there is a countable algebra $\cZ_0$
on~$\algb^+$ separating the points of~$\algb^+$ such that, for each $n\in \Nat$ and $P \sub \{1,\dots,n\}$,
the mapping $\zeta_{n,P}$ is $\otimes_n \sigma(\cZ_0)$-measurable.
\end{lem}
\begin{proof}
Since $|\algb^+|=\mathfrak{c}$ is a Kunen cardinal, each $\zeta_{n,P}$ is $\otimes_n \mathcal{P}(\algb^+)$-measurable.
Thus, we can find a countable family $\cC$ of subsets of~$\algb^+$ such that  
$\zeta_{n,P}$ is $\otimes_n \sigma(\cC)$-measurable for every $n\in \N$ and every $P \sub \{1,\dots,n\}$. Now, it is enough
to choose any countable algebra~$\cZ_0$ on~$\algb^+$
which separates the points of~$\algb^+$ and contains~$\mathcal{C}$.
\end{proof}

\begin{lem}\label{lem:NormRestriction}
Suppose $\mathfrak{c}$ is a Kunen cardinal.
Let $D$ be a finite partition of~$\algb^+$. Then the restriction of the supremum norm to~$A_D(K)$ is measurable
with respect to the trace of $\Ba(C(K),w)$ on $A_D(K)$.
\end{lem}
\begin{proof} Write $D=\{T_1,\dots,T_n\}$. We can suppose without loss of generality
(by enlarging~$\cZ$ if necessary) that $D\sub \cZ$ and that all functions $\zeta_{n,P}$ are $\otimes_n \sigma(\cZ)$-measurable
(see Lemma~\ref{lem:zetaP}). For each $P \sub \{1,\dots,n\}$, define
$$
	N_{D,P}:A_D(K) \to \{0,1\},
	\quad
	N_{D,P}(g):=
	\begin{cases}
	1 & \text{if $L_{D,P}(g)\neq 0$,}\\
	0 & \text{if $L_{D,P}(g)= 0$.}\\
	\end{cases}
$$
Since $L_{D,P}$ is measurable with respect to the trace of $\Ba(C(K),w)$ on~$A_D(K)$
(combine Lemma~\ref{lem:Phi} and Remark~\ref{remark:composition}), the same holds for~$N_{D,P}$.

Fix $g\in A_D(K)$ and write 
$$
	g=\sum_{r \sub s} y_r 1_{\widehat{J(r)}}
$$
for some $s\sub \algb^+$ finite with $|T_i\cap s|=1$ for every $i\in \{1,\dots,n\}$ and some collection of real numbers
$\{y_r:r\sub s\}$ (see the proof of Lemma~\ref{analysing:2}). 
Lemma~\ref{analysing:3}(iv) ensures that
$$
	y_r=\theta_{C_r}(g) \quad
	\mbox{for every }r \sub s \mbox{ with }J(r)\neq \0.
$$
Hence $g=\sum_{r \sub s} \theta_{C_r}(g) 1_{\widehat{J(r)}}$
and therefore
\begin{equation}\label{equation:SuperRepresentation}
	g=\sum_{r' \sub s} \left(\sum_{r \sub r'} \theta_{C_r}(g)\right) 1_{\widehat{W(r',s\sm r')}}
\end{equation}
(see again the proof of Lemma~\ref{analysing:2}). 
Write $T_i\cap s=\{b_i\}$ for every $i\in \{1,\dots,n\}$ and
observe that for each $P \sub \{1,\dots,n\}$
we have
$$
	N_{D,P}(g)=
	\begin{cases}
	1 & \text{if $\bigcap_{i\in P}b_i\neq \0$,}\\
	0 & \text{if $\bigcap_{i\in P}b_i= \0$.}
	\end{cases}
$$
Define $C(Q):=\{T_i:i\in Q\}$ for every $Q \sub \{1,\dots,n\}$. 
From~\eqref{equation:SuperRepresentation} it follows that
$$
	\|g\|=\sup_{\substack{r' \sub s\\J(r')\neq \0}}
	\left|\sum_{r \sub r'}\theta_{C_r}(g)\right|=
	\sup_{P \sub \{1,\dots,n\}} \left|\sum_{Q \sub P} \theta_{C(Q)}(g) \right|\cdot N_{D,P}(g).
$$
	
As $g\in A_D(K)$ is arbitrary, the norm function $\|\cdot\|$ coincides with 
$$
	\sup_{P \sub \{1,\dots,n\}}\left|\sum_{Q \sub P} \theta_{C(Q)}(\cdot) \right|\cdot N_{D,P}(\cdot)
$$
on~$A_D(K)$ and so $\|\cdot\|$ is measurable with respect to the trace of $\Ba(C(K),w)$ on $A_D(K)$. The proof is over.
\end{proof}

Finally, we arrive at the main result of this subsection:

\begin{theo}\label{analysing:6}
Suppose $\mathfrak{c}$ is a Kunen cardinal.
Then the restriction of the supremum norm to~$S(K)$ is measurable with respect to the trace of $\Ba(C(K),w)$
on $S(K)$.
\end{theo}
\begin{proof} 
Let $\Pi$ be the collection of all partitions of~$\algb^+$ into finitely many elements of~$\cZ$. 
By Lemma~\ref{lem:simplification}, we can write $S(K)=\bigcup_{D\in \Pi}A_D(K)$. Since $\Pi$ is countable and each $A_D(K)$ belongs to the trace of $\Ba(C(K),w)$
on~$S(K)$ (see Corollary~\ref{cor:SDmeasurable}), the result follows 
from Lemma~\ref{lem:NormRestriction}.
\end{proof}

\subsection{Some open problems}

\begin{enumerate}
\item[(A)] Let $L$ be a compact space. Is the $\Ba(C(L),w)$-measurability of the supremum norm on~$C(L)$
equivalent to the $w^*$-separability of~$B_{C(L)^*}$ or $C(L)^*$? What about the compact space~$K$ considered in this paper?
 
\item[(B)] Let $(X,\|\cdot\|)$ be a Banach space and suppose $A \sub X$ is a norm dense set (or linear subspace) such that 
$\|\cdot\|$ is relatively $\Ba(X,w)$-measurable on~$A$. Does this imply that $\|\cdot\|$ is $\Ba(X,w)$-measurable on~$X$?

Note that the analogous question for the property ``$B_{X^*}$ is $w^*$-separable'' has 
positive answer (just apply the Hahn-Banach theorem), while 
for the property ``$X^*$ is $w^*$-separable'' it has negative answer, 
see \cite[Example~1.1]{FPZ1997}.

\item[(C)] Let $L$ be a compact space and consider the `square' mapping
$$
	S:C(L) \to C(L), \quad S(g):=g^2.
$$
Which conditions on~$L$ ensure that $S$ is $\Ba(C(L),w)$-measurable? 

Note that if
$L$ carries a strictly positive measure, say~$\mu$, then
the supremum norm on~$C(L)$ can be computed as 
$$
	\|g\|=\lim_{n \to \infty} \left(\int_{L}f^{2n} \, d\mu\right)^{\frac{1}{2n}},
$$
hence $\|\cdot\|$ is $\Ba(C(L),w)$-measurable whenever $S$ is $\Ba(C(L),w)$-measurable.
What about the compact space~$K$ considered in this paper?
\end{enumerate}

\end{document}